\documentclass{amsart}

\usepackage{amssymb,amsmath,latexsym,amsthm}
\usepackage[english]{babel}

\newtheorem{theorem}{Theorem}[section]

\newtheorem{corollary}[theorem]{Corollary}

\theoremstyle{definition}

\numberwithin{equation}{section}

\renewcommand{\mod}[1]{\hspace{.05in} (\textrm{mod  } {#1})}

\newcommand{\qm}{\mathbb{Q}(\sqrt{m})}

\begin{document}

\title[]{Real quadratic fields in which every non-maximal order has relative ideal class number greater than one}

\author[Amanda {\sc Furness}]{{\sc Amanda} FURNESS}
\address{Amanda {\sc Furness}\\
Department of Mathematics,\\
Indiana University \\
Bloomington, IN 47405, USA}
\email{afurness@indiana.edu}

\author[Adam E. {\sc Parker}]{{\sc Adam E.} PARKER}
\address{Adam E. {\sc Parker}\\
Department of Mathematics and Computer Science\\
Wittenberg University, P.O. Box 720\\
Springfield, OH 45501, USA}
\email{aparker@wittenberg.edu}
\urladdr{http://userpages.wittenberg.edu/aparker}

\maketitle

\begin{abstract}
Cohn asked if for every real quadratic field $K=\qm$, with $m$ being the squarefree part of the field discriminant $d_0$ of $K$,  there exists a non-maximal order corresponding  to $f>1$ such that the relative ideal class number $h_{d_0}(f)=h(d_0 f^2)/h(d_0)$ is one.   We prove that for $\mathbb{Q}(\sqrt{46})$ there is no such order.  We also prove the analogous result for relative strict ideal class numbers and relative quadratic form class numbers.
\end{abstract}

\bigskip
\section{Introduction}

In his 1801 \emph{Disquisitiones arithmeticae} \cite{gauss}, Gauss established the theory of integral binary quadratic forms $Q(x,y) = a x^2 + b xy + cy^2$ under an equivalence  relation he called proper equivalence.  (This is equivalence under transformations of variables in $SL(2, \mathbb{Z}).$)  Gauss developed his theory only for quadratic forms of even middle coefficient $b= 2b'$ and he listed quadratic forms by values of their determinant $D=(b')^2 - a c.$  In this paper, it will be convenient to use the discriminant $d=4D = b^2 - 4 a c$, rather than the determinant $D$.  A discriminant $d$ is called a \emph{fundamental discriminant} or a \emph{field discriminant} if it is not divisible by another discriminant with the quotient being a perfect square.  Thus an arbitrary discriminant can be written $d= d_0 f^2$ where $d_0$ is a field discriminant.  The extension of Gauss's theory to allow an odd middle coefficient was done later.  Gauss called a quadratic form $Q(x,y)$ properly primitive if {\sc gcd}$( a,b,c)=1$.  He developed a theory of {composition of quadratic forms} that puts a group structure on the set of equivalence classes of properly primitive quadratic forms of a given discriminant $d$.    We let $H(d)$ denote the number of such equivalence classes and call $H(d)$ the \emph{quadratic form class number}.  Gauss also divided classes of quadratic forms into genera, whose number is always a power of $2$, depending on the prime factorization of $D$, which implies that $H(d)=H(4D)$ is divisible by this power of $2$.  Gauss computed tables of the number of quadratic form classes for different $D$, and for definite quadratic forms (namely $D<0$) he observed that the number of $D$ with one class per genus appears to be bounded, with the largest value found being $D=-1848.$  For indefinite quadratic forms (namely $D>0$) he observed (in Article 304) that there are many positive non-square $D$ which have one class per genus, and he commented that one could not doubt there are infinitely many such positive $D$.  

In 1858 Dirichlet \cite{dirichlet} addressed the comment of Gauss about indefinite quadratic forms.  He introduced the notion of \emph{relative quadratic form class number} $H_{d_0}(f),$ defined by
\[
H_{d_0}(f) := \frac{H(d_0 f^2)}{H(d_0)},
\]
and noted that $H_{d_0} (f)$ is an integer.  He gave a formula for it, and proved that there exist field discriminants $d_0$ for which there are infinitely many $f$ such that the relative quadratic form class number is $1$.  He also stated that one can furthermore find such values $d_0$ whose quadratic forms have one class per genus, and concluded by varying $f$  that there do exist infinitely many positive non-square $d$ for which the set of properly primitive quadratic forms of discriminant $d$ have one class per genus.  Thus he rigorously justified Gauss's comment.  Dirichlet did not give an explicit construction of a particular $d_0$.  In 1962 H. Cohn \cite[p. 130]{cohn1} computed some explicit cases, and in particular his results imply that for $d_0=5$ and $f=5^n$ for $n \geq 1$, one has $H(5)=1, H_5(5^n) =1$ and so $H(5^{2n+1})=1$.

It is natural to ask for which $d_0$ does Dirichlet's conclusion hold:  For which $d_0$ do there exist infinitely many $f$ with relative quadratic form class number $H_{d_0}(f) = 1$?  In 1962, Cohn \cite[p. 219]{cohn2} raised a closely related question concerning relative ideal class numbers $h_{d_0}(f)$ of orders in the real quadratic field $\mathbb{Q}(\sqrt{d_0})$ (defined below).  Cohn said that it is not known whether for each positive fundamental discriminant $d_0$ there always exists some $f>1$ such that the relative ideal class number $h_{d_0}(f)=1$.  This question was incorrectly attributed to Dirichlet in \cite{fp}.  The object of this paper is to answer Cohn's question, showing that there are some integers $d_0$ for which this cannot be done.  We shall simultaneously show that for the same $d_0$ the relative quadratic form class number $H_{d_0}(f)>1$ for all $f>1$.

We make use of the relation of classes of quadratic forms to strict invertible ideal classes in an order $\mathcal{O}_f$ of the quadratic field $\mathbb{Q}(\sqrt{d_0})$, which we review in Section 2.  For the quadratic field $\mathbb{Q}(\sqrt{d_0})$ and $f \geq 1$, the order $\mathcal{O}_f = \mathbb{Z}[1, f(\frac{d_0+\sqrt{d_0}}{2})]$ is a subring of the maximal order $\mathcal{O}_1$ and is a Dedekind domain if and only if $f=1$.  We let $h(d_0 f^2)$ denote the size of its group of invertible ideal classes under equivalence of ideals by principal ideals, and $h_+(d_0 f^2)$ be the order of its group under strict equivalence of ideals, which uses principal ideals of totally positive elements.  Here $h_+(d_0 f^2) = e h(d_0 f^2)$ where $e=1$ or $2$.  Parallel to relative quadratic form class numbers, we define the \emph{relative ideal class number}
\[
h_{d_0}(f) :=  \frac{h(d_0 f^2)}{h(d_0)}
\]
 and the \emph{relative strict ideal class number}
\[
h^+_{d_0} (f) := \frac{h_+(d_0 f^2)}{h_+(d_0)}.
\]

Our result will answer Cohn's question for all three of these quantities.

\begin{theorem}\label{1.1}
For each $f>1$, the three relative class numbers described above satisfy
\[
h_{184}(f)  = \frac{h(184f^2)}{h(184)} = \frac{h_+(184 f^2)}{h_+(184)}=\frac{H(184 f^2)}{H(184)} > 1.
\]
\end{theorem}

\section{ Background }
In this section we give background on the three types of class numbers ($H(d)$, $h_+(d)$ and $h(d)$) described in the introduction, as well as connections between them.  All results are taken from \cite{cohn2}.

Most well known is the case that $d=d_0$ is a field discriminant.  In this case there is a one-to-one correspondence between properly primitive quadratic form classes and ideal classes of the quadratic field $\mathbb{Q}(\sqrt{d_0})$ under \emph{strict} equivalence (also called \emph{narrow} equivalence) of ideals in the full ring of integers $\mathcal{O}_1 = \mathbb{Z}[1,\frac{d_0+\sqrt{d_0}}{2}]$.  Here two ideals $A, B$ are strictly equivalent if $(\lambda)A = (\mu) B$ for two principal ideals $(\lambda)$ and $(\mu)$ with $\frac{\lambda}{\mu}$ a totally positive element (meaning that itself and its algebraic conjugates are positive) or equivalently $N(\lambda \mu) >0$.  $N(\cdot)$ is the norm from the (real) quadratic field $K=\mathbb{Q}(\sqrt{d_0})$ to $\mathbb{Q}$.  This should be compared with \emph{ordinary} (or \emph{wide}) equivalence where there is no requirement of total positivity.  Following Cohn \cite[Chap. XII.3]{cohn2} we denote the number of strict classes by $h_+(d_0)$ which we call the \emph{strict ideal class number}, and we denote the number of ordinary classes by $h(d_0)$ which we call the \emph{ideal class number}.  The correspondence between properly primitive quadratic form classes and strict ideal classes described in  \cite[Chap. XII.6, Theorems 6 and 7]{cohn2}  gives 
\[
H(d_0) = h_+(d_0).
\]
Also note that $h_+(d_0)$ equals either $h(d_0)$ or $2h(d_0)$ as stated in \cite[Chapter XII.3, Theorem 3]{cohn2}.

Less well known is the case when $d=d_0 f^2$ is a non-fundamental discriminant.  The correspondence is now with invertible ideal classes with respect to the non-maximal order $\mathcal{O}_f = \mathbb{Z}[1, f(\frac{d_0+\sqrt{d_0}}{2})]$ of $\mathbb{Q}(\sqrt{d_0})$.  This ring is not a Dedekind domain, but it does contain a group of invertible ideals, which are those ideals having norm relatively prime to $f$.  One defines strict invertible ideal classes in this ring the same way as for the full ring of integers except we require all the ideals $A,B, \mu, \lambda$ to be invertible in $\mathcal{O}_f$.  Similarly for ordinary equivalence.  We define the \emph{strict invertible ideal class number} $h_+(d_0 f^2)$ to be the number of such strict invertible classes and the \emph{invertible ideal class number} $h(d_0 f^2)$ to be the number of ordinary invertible classes.  These definitions are detailed in \cite[Chap. XIII.2]{cohn2}.  Results analogous to those previously stated for fundamental discriminants are still true when we look at non-fundamental discriminants.   

\begin{theorem}\label{2.1} 
For $d=d_0 f^2$ with $d_0$ a field discriminant, the quadratic form class number $H(d)$ equals the strict ideal class number $h_+(d)$.
\end{theorem}
\begin{proof}
The proof is sketched in Cohn's book \cite{cohn2}.  The case $f=1$ is explicitly treated in Theorems 6 and 7 in Chap. XII.6.    The general case of $f>1$ is stated in Chap. XIII.2, p. 219, but the details are omitted.
\end{proof}

Since $H(d_0 f^2) = h_+( d_0 f^2)$ for $f\geq 1$, we see that the relative quadratic form class number $H_{d_0}(f)$ is equal to the relative strict ideal class number $h^+_{d_0}(f)$.   Because of Theorem \ref{2.3} below, we have a simple way to calculate $h_{d_0}(f)$, and so we want to relate $h^+_{d_0}(f)$ and $h_{d_0}(f)$.  The following theorem will be helpful in showing that connection.

\begin{theorem}\label{2.2}
For a real quadratic field with field discriminant $d_0$, we have
\[
h_+(d_0 f^2)  = \begin{cases} 
h(d_0 f^2) & \mbox{   if   } N(\varepsilon_{d_0 f^2})=-1,\\
2 h(d_0 f^2) & \mbox{   if   } N(\varepsilon_{d_0 f^2})=1,
\end{cases}
\]
where $\varepsilon_{d_0 f^2}$ is the fundamental unit in the order $\mathcal{O}_f$ of $\mathbb{Q}(\sqrt{d_0})$.
\end{theorem}

\begin{proof}  
The case of $f=1$ is given in Theorem 3 of \cite[p. 198]{cohn2}.  The details for $f>1$ are sketched in Chap. XIII.2.
\end{proof}

An immediate consequence is a condition for equality of the relative strict ideal class number and the relative ideal class number.
\begin{corollary}\label{cor}
The relative class numbers for real quadratic fields satisfy
\[
h^+_{d_0}(f) = 2 h_{d_0}(f) 
\]
if and only if $N(\varepsilon_{d_0}) = {-1}$ and $N(\varepsilon_{d_0 f^2}) =1 $.  In all other cases $
h^+_{d_0}(f) =  h_{d_0}(f).$
\end{corollary}
\begin{proof}
Certainly we will have equality if $N(\varepsilon_{d_0}) =N(\varepsilon_{d_0 f^2}) $.  Since $\varepsilon_{d_0 f^2}$ is also a unit in $\mathcal{O}_1$, we know that $\varepsilon_{d_0 f^2}=\varepsilon_{d_0}^k$ for some $k$ that depends on $d_0$ and $f$.  This  means that the case $N(\varepsilon_{d_0})=1$ and $N(\varepsilon_{d_0 f^2}) = -1$ never occurs.  The only other option is $
N(\varepsilon_{d_0})=-1$ and $N(\varepsilon_{d_0 f^2}) = 1$ (i.e. the $k$ above is even).  This case gives $h^+_{d_0}(f) = 2 h_{d_0}(f) $ by Theorem \ref{2.2}.
\end{proof}

We will use the following result which gives a formula for the relative class number $h_{d_0}(f)$.  If $d_0$ is a field discriminant, we let $m$ denote its squarefree part, so that
\[
m=\begin{cases}
d_0 & \mbox{ if } \, m \equiv 1 \mod 4, \\
\frac{d_0}{4} & \mbox{ if } \, m \equiv 2,3 \mod 4.
\end{cases}
\]

\begin{theorem}\cite{cohn1}\cite[p. 217]{cohn2}\label{2.3}  Let $d_0$ be a field discriminant, with $m$ its squarefree part, so that $K=\qm$.
Let $\varepsilon_m$ be the fundamental unit of $K$ written as
\[
\varepsilon_m = \frac{x+y \sqrt{m}}{c} \, \mbox{ where } \, c=\begin{cases} 2 &\mbox{  if  } \,\,m \equiv  1 \mod 4\\
1 & \mbox{  if  }\,\, m \equiv 2,3 \mod 4.
\end{cases}
\]
Define
\[  
\psi(f) = f \displaystyle \prod_{q \mid f} \left( 1- \left( \frac{d_0}{q} \right) \frac{1}{q} \right)
 \]
where $\displaystyle \left( \frac{d_0}{q} \right)$ is the Legendre symbol and $q$ is prime. Define $\phi(f)$ to be the 
smallest positive integer such that $(\varepsilon_m)^{\phi(f)}$ is in a non-maximal order $\mathcal{O}_{f} ,$  (namely that $\phi(f)$ is the smallest integer such that $ (\varepsilon_m)^{\phi(f)} = \displaystyle \frac{a+b\sqrt{m}}{c}$ where $b \equiv 0 \mod f$).
Then $$ h_{d_0}(f) = \frac{\psi(f)}{\phi(f)}. $$ 
\end{theorem}

Please note that because we write $\varepsilon_m$ in the form where we divide by $c$, the ``y coordinate" of the fundamental unit will always be an integer, and so it makes sense to ask if $f$ divides $y$.

\section{Results}

A consequence of Theorem \ref{2.3} is the following, which gives a criterion for when the relative ideal class number will be equal to one.

\begin{theorem}\label{3.1}
Let $K=\qm$  be a real quadratic field, with $m$  being the squarefree part of the field discriminant $d_0$ of $K$.  If the fundamental unit $\varepsilon_m=\frac{x+y\sqrt{m}}{c}$ (using the notation of  Theorem \ref{2.3}) has the property that $m$ does not divide $y$, then there is some prime $f>1$ such that $h_{d_0}(f) =1.$
\end{theorem}

\begin{proof}  If $m$ does not divide $y$ then there is a prime $f$ that divides $m$ but not $y$ (because $m$ is squarefree).  Choose such an $f$.  On the one hand, since $f|m$ and $m|d_0$ we know the Legendre symbol is $0$, and hence $\psi(f) = f$.  On the other hand, since $h_d(f)$ is an integer, $ \phi(f) = 1$ or $f$.  And since we chose $f$ not dividing $y$, we know $(\varepsilon_m)^1 \notin \mathcal{O}_f$.  So $\phi(f)=f$ and $h_{d_0}(f)=1$.  
\end{proof}

In the case of $m=46$ it is impossible to find such an $f$ since the fundamental unit is $\varepsilon_{46}=24335+3588 \sqrt{46}$ and $46  | 3588$.  This is the first fundamental unit with the property that $m$ divides $y$, and so it is a natural candidate to consider more carefully.

We will now prove Theorem \ref{1.1}.

\begin{proof}

For all positive fundamental discriminants $d_0$ with squarefree part $1<m<46$ there is some relative ideal class number $h_{d_0}(f)=1$ by the criterion of Theorem \ref{3.1}.  We now treat the case $m=46$.  

Since $m\equiv 2 \mod{4}$, $d_0=4m = 184$ and $c=1$.  We will start by considering primes $f$ which do not divide $46$.  Because $f$ is prime, $\psi(f) = f-\left(\frac{184}{f}\right) = f-\left(\frac{46}{f}\right)$ is even.

We also know that the fundamental unit is $\varepsilon_{46}=24335+3588 \sqrt{46}$.  This is the smallest solution to the Pell equation $x^2-46 y^2 = 1$.  If we write $\varepsilon_{46}=a_1+b_1 \sqrt{46}$, then  other solutions are of the form $(a_n+b_n \sqrt{46})=(a_1+b_1 \sqrt{46})^n$.  These obey the recurrence relations

\begin{align}\label{e3}
a_{n+1} &= a_1 a_n + 46 b_1 b_n, \\
b_{n+1} &= a_1 b_n + b_1 a_n. \nonumber
\end{align}

We can write this in matrix notation as
\[
\left(  \begin{array}{cc}
a_1 & 46 b_1  \\
b_1 & a_1 
\end{array} \right) 
\left( \begin{array}{c}
a_n \\
b_n
\end{array} \right) = 
\left( \begin{array}{c}
a_{n+1} \\
b_{n+1}
\end{array}\right)
\]
or
\[
\left(  \begin{array}{cc}
a_1 & 46 b_1  \\
b_1 & a_1 
\end{array} \right)^n
\left( \begin{array}{c}
a_1 \\
b_1
\end{array} \right) = 
\left( \begin{array}{c}
a_{n+1} \\
b_{n+1}
\end{array}\right)
\]
or

\begin{equation}\label{e1}
\left(  \begin{array}{cc}
a_1 & 46 b_1  \\
b_1 & a_1 
\end{array} \right)^{n-k}
\left( \begin{array}{c}
a_k \\
b_k
\end{array} \right) = 
\left( \begin{array}{c}
a_{n} \\
b_{n}
\end{array}\right).
\end{equation}

It will be helpful to recognize that all powers of this matrix have determinant $1$ (since $a_1^2-46 b_1^2 = 1)$ and have a nice ``almost diagonal" form:
\[
\left(  \begin{array}{cc}
a_1 & 46 b_1  \\
b_1 & a_1 
\end{array} \right)^t=
\left(  \begin{array}{cc}
A & 46 B \\
B & A
\end{array} \right)
\]
for some integers $A,B$.

We start by examining $(\varepsilon_{46})^{\psi(f)}$.  We know $\phi(f)$ -  the minimum exponent such that $(\varepsilon_{46})^{\psi(f)} \in \mathcal{O}_f$ - must divide $\psi(f)$ because $h_d(f) \in \mathbb{Z}$.  Therefore, $(\varepsilon_{46})^{\psi(f)} \in \mathcal{O}_f$.  Said differently, for

\[
\left(  \begin{array}{cc}
a_1 & 46 b_1  \\
b_1 & a_1 
\end{array} \right)^{\psi(f)-1}
\left( \begin{array}{c}
a_1 \\
b_1
\end{array} \right) = 
\left( \begin{array}{c}
a_{\psi(f)} \\
b_{\psi(f)}
\end{array}\right),
\]
 we know $f|b_{\psi(f)}$.  Then our goal will be to show $f|b_n$ for some $n<\psi(f)$.  

We look $\mod{f}$.  Menezes in \cite[p. 59]{menezes} discusses this in a slightly different way.  He considers $\emph{C}$ to be the set of all solutions $(x,y) \in F_q \times F_q$ to the Pell equation $x^2-Dy^2=1$ where $q$ is a power of an odd prime.  He defines a group structure by 

\begin{equation}\label{e2}
(x_1, y_2) \oplus (x_2,y_2) = (x_1 x_2+Dy_1y_2 \,, \, x_1y_2+x_2 y_1).
\end{equation}
Then his Lemma 4.4 states that $(C, \oplus)$ is an abelian group with identity $(1,0)$ and his Theorem 4.5 states that $(C, \oplus)$ is a cyclic group of order $q-\left(\frac{D}{q}\right),$ (which is our $\psi(f)$ when $D=46$ and $q=f$).

Notice that the group structure given by  (\ref{e2}) is exactly the structure given by relations (\ref{e3}) and the subsequent matrix multiplications.  Therefore, 
\[
\left(  \begin{array}{cc}
a_1 & 46 b_1  \\
b_1 & a_1 
\end{array} \right)^{\psi(f)} = \left(  \begin{array}{cc}
1 & 0  \\
0 & 1
\end{array} \right) \mod{f}.
\]

It turns out that  $\left(  \begin{smallmatrix}
a_1 & 46 b_1  \\
b_1 & a_1 
\end{smallmatrix} \right)^{\psi(f)/2} \mod{f}$ is diagonal.  To see this, we write

\[\left(  \begin{array}{cc}
a_1 & 46 b_1  \\
b_1 & a_1 
\end{array} \right)^{\psi(f)/2} = \left(  \begin{array}{cc}
A & 46 B  \\
B & A
\end{array} \right). 
\]
 Then

\[
\left(\begin{array}{cc}
A & 46 B  \\
B & A 
\end{array} \right)
\left(\begin{array}{cc}
A & 46 B  \\
B & A 
\end{array} \right) = 
\left(\begin{array}{cc}
A^2+46  B^2 & 2\cdot 46 A B  \\
2 A B & A^2+46  B^2
\end{array} \right)=
\left(\begin{array}{cc}
1 & 0  \\
0 & 1
\end{array} \right) \mod{f}.
\]
Since $f$ was chosen not to divide $46$, either $f$ divides $A$ or $B$.  If $f| A$, then $46 B^2 \equiv 1 \mod{f}$. But at the same time, this matrix should have determinant $A^2-46 B^2 = 1$, which means $-46B^2\equiv 1 \mod{f}$, which gives a contradiction.   Therefore $f|B$ and we see that 
\[
\left(  \begin{array}{cc}
a_1 & 46 b_1  \\
b_1 & a_1 
\end{array} \right)^{\psi(f)/2} = \left(  \begin{array}{cc}
A & 0  \\
0 & A
\end{array} \right) \mod{f}.
\]

Setting $n=\psi(f)$ and $k=\psi(f)/2$ in equation (\ref{e1}) above, we obtain the equation

\[
\left(  \begin{array}{cc}
a_1 & 46 b_1  \\
b_1 & a_1 
\end{array} \right)^{\psi(f)/2}
\left( \begin{array}{c}
a_{\psi(f)/2}\\
b_{\psi(f)/2}
\end{array} \right) = 
\left( \begin{array}{c}
a_{\psi(f)} \\
b_{\psi(f)}
\end{array}\right).
\]
Multiplying on the left by $\left(  \begin{smallmatrix}
a_1 & 46 b_1  \\
b_1 & a_1 
\end{smallmatrix} \right)^{\psi(f)/2}$ gives the following equations:

\[
\left(  \begin{array}{cc}
a_1 & 46 b_1  \\
b_1 & a_1 
\end{array} \right)^{\psi(f)}
\left( \begin{array}{c}
a_{\psi(f)/2}\\
b_{\psi(f)/2}
\end{array} \right) = 
\left(  \begin{array}{cc}
a_1 & 46 b_1  \\
b_1 & a_1 
\end{array} \right)^{\psi(f)/2}
\left( \begin{array}{c}
a_{\psi(f)} \\
b_{\psi(f)}
\end{array}\right).
\]

\[
\left(  \begin{array}{cc}
1 & 0 \\
0 & 1
\end{array} \right)
\left( \begin{array}{c}
a_{\psi(f)/2}\\
b_{\psi(f)/2}
\end{array} \right) = 
\left(\begin{array}{c}
a_{\psi(f)/2} \\
b_{\psi(f)/2}
\end{array}\right)=
\left(  \begin{array}{cc}
A & 0 \\
0 & A
\end{array} \right)
\left( \begin{array}{c}
a_{\psi(f)} \\
b_{\psi(f)}
\end{array}\right)  \mod{f}.
\]

Therefore $b_{\psi(f)/2} \equiv A b_{\psi(f)} \mod{f}$, and since we already know that $f| b_{\psi(f)}$ we know that $(\varepsilon_{46})^{\psi(f)/2} \in \mathcal{O}_f$.  Therefore the relative ideal class number satisfies $h_{d_0}(f) \geq 2$, answering the question of Cohn.

A computation shows that when $m=46$ $h_{d_0}(23)=23$ and $h_{d_0}(2)=2$, because in both these cases, $\varepsilon_{46} \in \mathcal{O}_f$.

Therefore, for all primes $f$, $h_{d_0}(f) > 1$.  And since  $f | g$ implies $h_{d_0}(f)|h_{d_0}(g)$, this proves that $h_{d_0}(f)>1$ for all $f>1$.

To show the connection with the other relative class numbers we need only notice that $N(\varepsilon_{184})=1$.  Then by the analysis in Corollary \ref{cor}, we see that for all $f>1$ we have $N(\varepsilon_{184 f^2})=1$ and  so
\[
h^+_{d_0}(f) = \frac{h_+(d_0 f^2)}{h_+(d_0)} = \frac{2 h(d_0 f^2)}{2 h(d_0)} = h_{d_0}(f),
\]
as desired.  Since $H_{d_0}(f) = h^+_{d_0}(f)$, by Theorem \ref{2.1} we see that all three relative class numbers, $H_{d_0}(f)$, $h^+_{d_0}(f),$ and $h_{d_0}(f)$, are all equal and are all $>1$ for all $f >1$.

\end{proof}

One might ask how many other quadratic fields $\qm$ satisfy the property that every non-maximal order gives a relative ideal class number $>1.$   By Theorem \ref{3.1}, it  suffices to only consider cases where $m$ divides the $y$ coordinate of $\varepsilon_{m}$.  Quadratic fields with this property were studied in \cite{stephens} while researching powerful numbers.  They tested all $\qm$ with $m<10^7$ and found only $8$  fields such that $m$ divides $y$.  They are $m=46$, $430$, $1817$, $58254$, $209991$, $1752299$, $3124318$  and $4099215$.  Hence for all other quadratic fields $\qm$ with $m$ squarefree and  $<10^7$ one can find a prime $f$ that divides $m$ but not $y$, and so by Theorem \ref{3.1} we have $h_{d_0}(f)=1$.  While our proof of Theorem 1.1 is specific to $m=46$, the arguments can be adapted to show that the above seven other quadratic fields where $m$ divides the $y$ coordinate of $\varepsilon_{m}$ will never have a (non-maximal) relative ideal class number equal to $1$.  

There appear to be few general results for when $m$ divides (or does not divide) $y$.  Ankeny, Artin and Chowla in \cite{aac}  asked if $p$ is a prime with $p \equiv 1 \mod 4$ and if $(x+y\sqrt{p})/2$ is the fundamental unit of $\mathbb{Q}(\sqrt{p})$, then does $p \nmid y$?  Mordell in \cite{mordell}  conjectured the same is true for $p\equiv 3 \mod 4$.  Neither of these have been proven, though both have been checked for large primes \cite{stephens} \cite{poorten}.  It seems to be an open problem whether there exist infinitely many squarefree $m$ such that the fundamental unit $\varepsilon_m = \frac{x+y\sqrt{m}}{c}$ of $\qm$ given as in Theorem \ref{2.3} has $m|y$.

\subsection{Acknowledgements}
The authors would like to thank the 2009 REU (Research Experience for Undergraduates) at Mount Holyoke College, in particular Giuliana Davidoff, for bringing relative class numbers to our attention.  We would also like to thank Jeffrey Lagarias for his comments that greatly improved this paper.  Finally we want to sincerely thank the generous anonymous referee who spent much energy on this paper, in particular explaining the connection between the quadratic form class numbers and ideal class numbers.

\end{document}